\documentclass[a4paper,11pt]{amsart}
\usepackage[margin=1in]{geometry}
\usepackage{graphicx}
\usepackage{amsmath}
\usepackage[utf8]{inputenc}
\usepackage{amsfonts}
\usepackage{amsthm}
\usepackage{xcolor}
\usepackage[bookmarksopen=true,colorlinks=true, linkcolor=red, citecolor=blue]{hyperref}
\usepackage{marginnote}

\numberwithin{equation}{section}
\newtheorem{theorem}{Theorem}[section]
\newtheorem{proposition}[theorem]{Proposition}
\newtheorem{lemma}[theorem]{Lemma}
\newtheorem{corollary}[theorem]{Corollary}
\newtheorem{remark}[theorem]{Remark}
\newtheorem{definition}{Definition}

\newcommand{\set}[1]{\left\lbrace #1\right\rbrace}
\providecommand{\abs}[1]{\left\lvert#1\right\rvert}
\providecommand{\norm}[1]{\left\lVert#1\right\rVert}

\begin{document}

\title[Stabilization of the Gear--Grimshaw system]{Stabilization of the Gear--Grimshaw system with weak damping}
\author[Capistrano--Filho]{R. A. Capistrano--Filho}
\address{Departamento de Matemática, Universidade Federal de Pernambuco, UFPE, CEP 50740-545, Recife, PE, Brazil.}
\email{capistranofilho@dmat.ufpe.br}

\date{Version 2016-12-02-a}

\begin{abstract}
The aim of this work is to consider the internal stabilization of a nonlinear coupled system of two Korteweg--de Vries equations in a finite interval under the effect of a  very weak localized damping. The system was introduced by Gear and Grimshaw to model the interactions of two-dimensional, long, internal gravity waves propagation in a stratified fluid. Considering feedback controls laws and using \textit{"Compactness--Uniqueness Argument"}, which reduce the problem to use a unique continuation property, we establish the exponential stability of the weak solutions when the exponent in the nonlinear term ranges over the interval $[1,4)$.
\end{abstract}

\maketitle

\section{Introduction\label{Sec0}}
\subsection{Setting of the problem} In \cite{geargrinshaw}, a complex system of equations was derived by Gear and Grimshaw to model the strong interaction of two-dimensional, long, internal gravity waves propagating on neighboring pycnoclines in a stratified fluid. It has the structure of a pair of Korteweg-de Vries equations coupled through both dispersive and nonlinear effects and has been the object of intensive research in recent years. In particular, we also refer to \cite{bona} for an extensive discussion on the physical relevance of the system.

An interesting possibility now presents itself is the study of the stability properties when the model is posed on a bounded domain $(0,L)$. In this paper, we are mainly concerned with the study of the Gear-Grimshaw system
\begin{equation}
\label{gg1}
\begin{cases}
u_t + uu_x+u_{xxx} + a_3v_{xxx} + a_1vv_x+a_2 (uv)_x =0, & \text{in} \,\, (0,L)\times (0,\infty),\\
c v_t +rv_x +vv_x+a_3b_2u_{xxx} +v_{xxx}+a_2b_2uu_x+a_1b_2(uv)_x  =0,  & \text{in} \,\, (0,L)\times (0,\infty), \\
u(x,0)= u^0(x), \quad v(x,0)=  v^0(x), & \text{in} \,\, (0,L),
\end{cases}
\end{equation}
satisfying the following boundary conditions
\begin{equation}\label{gg2}
\begin{cases}
u(0,t)=0,\,\,u(L,t)=0,\,\,u_{x}(L,t)=0, & \text{in} \,\, (0,\infty),\\
v(0,t)=0,\,\,v(L,t)=0,\,\,v_{x}(L,t)=0, & \text{in} \,\, (0,\infty),
\end{cases}
\end{equation}
where $a_1, a_2, a_3, b_2, c, r\in \mathbb{R}$. We also assume that
$$1-a_3^2 b_2 > 0 \quad \text{and} \quad  b_2, c > 0.$$

The purpose is to see whether one can force the solutions of those systems to have certain desired properties by choosing appropriate damping mechanism. More precisely, we study the following fundamental problem related to the asymptotic behavior of the solutions for $t$ sufficiently large:

\vglue 0.1 cm

\noindent\textbf{Stabilization problem: }\textit{Can one find two feedback controls laws: $f=Gu$ and $g=Gv$ so that the system
\begin{equation}\label{1-1}
\left\{
\begin{array}{l}
\vspace{1mm}u_t+u_{xxx}+a_3v_{xxx}+a(u)u_x + a_1vv_x + a_2(uv)_x + Gu =  0 \\
b_1v_t + rv_x + v_{xxx} + b_2a_3u_{xxx} + a(v)v_x + b_2a_2uu_x + b_2a_1(uv)_x + Gv = 0,
\end{array}
\right.
\end{equation}
with boundary condition \eqref{gg2}, is asymptotically stable as $t\rightarrow + \infty$ ?}

\vglue 0.1 cm
\noindent If such the feedback controls laws exists, then the system \eqref{1-1}-\eqref{gg2} is said to be \textbf{stabilizable}.

\subsection{State of art} In what concerns the stabilization problems, most of the works have been focused on a bounded interval with a localized internal damping (see, for instance, \cite{PazSou} and the references therein). However, the stabilization results for system \eqref{gg1}-\eqref{gg2} was first obtained in \cite{Dav3}, when the authors considered the system in a periodic domain. Recently, Capistrano--Filho \textit{et al.} \cite{cf-pazoto-komornik} proved a result which extends the result proved by Dávila \cite{Dav3}, which one was proved only for $s\le 2$. More precisely, in \cite{cf-pazoto-komornik}, they showed that for any fixed  integer $s\ge 3$, the solutions are exponentially stable in the Sobolev spaces
\begin{equation*}
H^{s}_{p}(0,1):=\set{u\in H^s(0,1)\ :\ \partial^n_x u(0)=\partial_x^n u(1),\quad n=0,\ldots,s}
\end{equation*}
with periodic boundary conditions.

\vglue 0.2 cm
\noindent
{\bf Theorem $\mathcal{A}$ }(Capistrano--Filho \textit{et al.} \cite{cf-pazoto-komornik}) {\em Consider
\begin{equation}\label{13}
b_1=b_2=1
\end{equation}
and 
\begin{equation}\label{14}
r=0,\quad a_1^2+a_2^2=a_1+a_2,\quad \abs{a_3}<1,\quad\text{and}\quad  (a_1-1)a_3=(a_2-1)a_3=0.
\end{equation}

If $\phi,\psi\in H^{s}_{p}(0,1)$ for some integer $s\ge 3$, then the solution of 
\begin{equation}\label{11}
\begin{cases}
u_t+uu_x+u_{xxx}+a_3v_{xxx}+a_1vv_x+a_2(uv)_x+k(u-[u])=0,\\
v_t+vv_x+v_{xxx}+a_3u_{xxx}+b_2a_2uu_x+b_2a_1(uv)_x+k(v-[v])=0,\\
u(0,x)=\phi(x),\\
v(0,x)=\psi(x),
\end{cases}
\end{equation}
where $[f]$ denotes the mean value of $f$ defined by
\begin{equation*}
[f]:=\int_0^1f(x)\ dx
\end{equation*}
satisfies the estimate
\begin{equation*}
\norm{u(t)-[u(t)]}_{H^{s}_{p}(0,1)}+\norm{v(t)-[v(t)]}_{H^{s}_{p}(0,1)}=o\left( e^{-k't}\right) ,\quad \text{ as } \quad  t\to\infty
\end{equation*}
for each $k'<k$.
}

\vglue 0.1 cm

The proof of Theorem $\mathcal{A}$  follows the ideas introduced in \cite{KomRusZha1991} for the usual KdV equation by using the infinite family of conservation laws for this equation. Such conservations lead to the construction of a suitable Lyapunov function that gives the exponential decay of the solutions.

Concerning with bounded domain $(0,L)$, recently, Nina \textit{et. al.} \cite{nina} proved that under presence of a localized damping, represented by a function $b=b(x)$, the following system  
\begin{equation}\label{n1}
\begin{cases}
u_t+u_{xxx}+a_3v_{xxx}+a(u)u_x + a_1vv_x + a_2(uv)_x +b(x)u=0, \\
b_1v_t + rv_x + v_{xxx} + b_2a_3u_{xxx} + a(v)v_x + b_2a_2uu_x + b_2a_1(uv)_x +b(x)u=0, \\
u(x,0)= u^0(x), \quad v(x,0)=  v^0(x),
\end{cases}
\end{equation}
where $0<x<L$, $t>0$, with boundary conditions \eqref{gg2} is globally uniformly exponential stable when $b$ satisfies
\begin{equation}\label{b}
\left\{
\begin{array}{l}
\vspace{1mm} b \in L^2 (0,L) \,\hbox { is a nonnegative function, such that }\\
b(x) \geq b_0 > 0 \, \hbox {
a. e. in } \, \omega, \mbox{ where } \omega\subset (0,L) \mbox{ is a
nonempty open set. }
\end{array}
\right.
\end{equation}
More precisely, they proved the following result:

\vglue 0.2 cm
\noindent
{\bf Theorem $\mathcal{B}$ }(Nina \textit{et al.} \cite{nina}) {\em Let $a=a(x)$ be a $\mathcal{C}^2$ function such that
$$|a(x)|\leq C(1+|x|^p),\,\,|a'(x)|\leq C(1+|x|^{p-1}),\,\,|a''(x)|\leq C(1+|x|^{p-2}), \quad\forall\,x\in \mathbb{R}$$
where $C$ is a positive constant and $1\leq p < 4$. Then, if $b$ satisfies \eqref{b}, system \eqref{n1}-\eqref{gg2} is globally uniformly exponential stable.}

\vglue 0.1 cm

The techniques used to prove this result are different from those used in the proof of Theorem $\mathcal{A}$. The proof of the Theorem $\mathcal{B}$ is reduced to show a \textit{unique continuation property} one since $b(x)u=b(x)v=0$ implies that $(u,v)\equiv (0,0)$ in $\{b(x)>0\}\times (0,T)$. However, in this problem, the unique continuation property can not be applied directly. To overcome this problem the authors developed a Carleman inequality which allows the authors prove directly the unique continuation of weak solution.

\subsection{Main result} In this work we treat a very special case in which the source terms $f$ and $g$, introduced in the equation \eqref{1-1}, are defined by the operators
\begin{equation}\label{operator_new}
Gu=1_{\omega}\left(u(t,x)-\frac{1}{|\omega|}\int_{\omega}u(t,x)dx\right)
\end{equation}
and
\begin{equation}\label{operator_new_1}
Gv=1_{\omega}\left(v(t,x)-\frac{1}{|\omega|}\int_{\omega}v(t,x)dx\right),
\end{equation}
respectively. Here $\omega\subset(0,L)$ is a nomempty open set and $1_{\omega}$ denotes the characteristic function of the set $\omega$. 

We assume that $a=a(s)$ is real-valued function that satisfying the condition
\begin{equation}\label{a}
\left\{
\begin{array}{l}
\vspace{1mm}a(0)=0, \quad |a^{(j)}(s)|\leq c\,(1 + |s|^{p-j}), \forall \,s\in
\mathbb{R},\\
\vspace{1mm}\mbox{where } c \mbox{ is a positive constant and } j=0,1\, \mbox{ if }\, 1\leq p < 2\\
\mbox{and } j=0,1,2\, \mbox{ if }\, p\geq 2.
\end{array}
\right.
\end{equation}

Let us consider the total energy associated to (\ref{1-1}), in this case
\begin{equation}\label{6a}
E_s(t) = \frac{1}{2}\int_0^L (b_2u^2 + b_1v^2) dx.
\end{equation}
Then, we can (formally) verify that

\begin{equation}\label{dissipation}
\begin{array}{l}
\displaystyle\frac{1}{2}\frac{d}{dt}\int_0^L (b_2u^2 + b_1v^2) dx = -\left[\frac{b_2}{2}u_x^2(0,t) + \frac{1}{2}v_x^2(0,t) + a_3b_2u_x(0,t)v_x(0,t)\right]\\-\left(b_2||Gu||^2_{L^2(\omega)} + ||Gv||^2_{L^2(\omega)}\right) = -\displaystyle\frac{1}{2}\left(\sqrt{b_2}u_x(0,t)+\sqrt{a^2_3b_2}v_x(0,t)\right)^2
\\ -\vspace{1mm}\displaystyle\frac{1}{2}(1-a^2_3b_2)v^2_x(0,t)-\left(b_2||Gu||^2_{L^2(\omega)} + ||Gv||^2_{L^2(\omega)}\right)\leq 0,
\end{array}
\end{equation}
for any $t>0$. The inequality above shows that the terms $Gu$ and $Gv$ plays the role of a damping mechanisms and, consequently, we can investigate whether the solutions of \eqref{1-1}-\eqref{gg2} tend to zero as $t\rightarrow \infty$ and under what rate they decay.

Thus, the main result of this work gives a answer to the stabilization problem proposed on the beginning of this paper and can be read as follows.

\begin{theorem}\label{main-dec} Let $a=a(x)$ be a $\mathcal{C}^2$ function such that
$$|a(x)|\leq C(1+|x|^p),\,\,|a'(x)|\leq C(1+|x|^{p-1}),\,\,|a''(x)|\leq C(1+|x|^{p-2}), \quad\forall\,x\in \mathbb{R}$$
where $C$ is a positive constant and $1\leq p < 4$. Then, there exist positive constants $C$ and $k$, such that for any
$(u^0,v^0)\in [L^2(0,L)]^2$ with
\[
E_s(0)\leq R_0,
\]
the corresponding solution $(u,v)$ of
\begin{equation}\label{1-1a}
\left\{
\begin{array}{l}
\vspace{1mm}u_t+u_{xxx}+a_3v_{xxx}+a(u)u_x + a_1vv_x + a_2(uv)_x + Gu =  0 \\
b_1v_t + rv_x + v_{xxx} + b_2a_3u_{xxx} + a(v)v_x + b_2a_2uu_x + b_2a_1(uv)_x + Gv = 0,
\end{array}
\right.
\end{equation}
satisfying the following boundary conditions
\begin{equation}\label{gg2a}
\begin{cases}
u(0,t)=0,\,\,u(L,t)=0,\,\,u_{x}(L,t)=0,\\
v(0,t)=0,\,\,v(L,t)=0,\,\,v_{x}(L,t)=0,
\end{cases}
\end{equation}
where $0<x<L$, $t>0$, $Gu$ and $Gv$ are defined by \eqref{operator_new} and \eqref{operator_new_1}, respectively, satisfies
\begin{equation}
E_s(t)\leq Ce^{-kt}E_s(0)\text{, }\quad \forall t\geq 0.
\label{m3}
\end{equation}
\end{theorem}

To prove Theorem \ref{main-dec} we use the so called "Compactness-Uniqueness Argument" which reduces our problem to use a unique continuation property proved by Nina \textit{et al.} \cite{nina} (for more details see Section \ref{sec3}).

\vglue 0.2 cm

The following remarks are now in order:

\begin{remark}
A similar feedback law was used in \cite{Russell1} and, more recently, in \cite{Laurent} for Korteweg-de Vries equation (KdV), to prove a globally uniformly exponential result in a periodic domain. In \cite{Laurent} the damping with a null mean was introduced to conserve the integral of the solution, which for KdV represents the mass (or volume) of the fluid. Such form is thus motivated by a physical interpretation.
\end{remark}

\begin{remark}\label{rmk}
Note that Theorem \ref{main-dec} improves the result proved in \cite{nina} (see Theorem $\mathcal{B}$) in the sense that the decay is obtained with a weaker amount of damping, which not involves a physical interpretation. Moreover, Theorem \ref{main-dec} still holds true for other types of feedback laws, for instance, if the feedback is defined by $f=1_{\omega}(-\frac{d^2}{d x^2} u)^{-1}$ and $g=1_{\omega}(-\frac{d^2}{d x^2} v)^{-1}$. This previous damping mechanism was used by Massarolo \textit{et al.} \cite{massarolo-menzala-pazoto1} for the KdV equation.
\end{remark}


\medskip

The paper is outlined as follows:

\medskip

---- Section \ref{sec2} we review some results of the existence of solutions of the system \eqref{1-1a}-\eqref{gg2a} proved in \cite{nina} that will be used thereafter. In addition, we prove that the nonlinear problem with the extra terms $Gu$ and $Gv$ is also well-posedness. 

\medskip

---- Section \ref{sec3} is devoted to prove our main result, Theorem \ref{main-dec}.

\medskip

----  Section \ref{sec4} is related with some extension results. More precisely, when $a(x)=x^4$ the stabilization of the solution of the system \eqref{1-1a}-\eqref{gg2a} is obtained. Moreover, for $a(u)=u^p$, $p\in(2,4)$, the existence of weak solutions is also verified in this section. 

\medskip 

----  Finally, Section \ref{sec5} contains further comments and a open problem related with the solutions of the system \eqref{1-1a}-\eqref{gg2a}.

\section{Existence of solutions for the Gear-Grimshaw system\label{sec2}}

Most of the results in this section were proved by Nina \textit{et al.} \cite{nina} and Rosier \textit{et al.} \cite{Zhang}. To make the work more complete we present all of this results and, additionally, we prove the well-posedness result for the nonlinear system with extra terms $Gu$ and $Gv$.

\subsection{The linear system}
In this subsection we present results concerning of the existence of solutions of the linear system corresponding to \eqref{1-1a}-\eqref{gg2a}, namely
\begin{equation}
\label{linear}
\begin{cases}
u_t + u_{xxx} + a_3v_{xxx} =0, & \text{in} \,\, (0,L)\times (0,\infty),\\
c v_t +rv_x +a_3b_2u_{xxx} +v_{xxx}=0,  & \text{in} \,\, (0,L)\times (0,T), \\
u(0,t)=0,\,\,u(L,t)=0,\,\,u_{x}(L,t)=0,  & \text{in} \,\, (0,\infty), \\
v(0,t)=0,\,\,v(L,t)=0,\,\,v_{x}(L,t)=0,  & \text{in} \,\, ((0,\infty), \\
u(x,0)= u^0(x), \quad v(x,0)=  v^0(x), & \text{in} \,\, (0,\infty).
\end{cases}
\end{equation}
First, we need introduce the Hilbert space $X=[L^{2}(0,L)]^2$ endowed with the inner product
\begin{eqnarray}
    \left( (u,v),(\varphi,\psi)\right)_X=\frac{b_{2}}{b_{1}}\int^{L}_{0}u\varphi dx+\int^{L}_{0}v\psi dx \nonumber
\end{eqnarray}
and consider the operator
$$\mathcal{A}:D(\mathcal{A})\subset X \rightarrow X$$
where
$$D(\mathcal{A})=\left\{(u,v)\in [H^3(0,L)]^2\,:\, u(0)=v(0)=u(L)=v(L)=u_x(L)=v_x(L)=0\right\}$$
and
\begin{equation}\label{operator}
    \mathcal{A}(u,v)=\left(
\begin{array}{c}
    -u_{xxx}-a_{3}v_{xxx}\\
    -\displaystyle\frac{r}{b_{1}}v_{x}-\frac{b_{2}a_{3}}{b_{1}}u_{xxx}-\frac{1}{b_{1}}v_{xxx}
\end{array}
\right).
\end{equation}

With this notation, system \eqref{linear} can be now written as an abstract Cauchy problem in $X$. Letting $U=(u,v)$ we have
 \begin{equation*}
 \begin{cases}
\displaystyle\frac{dU}{dt}=\mathcal{A}U\\
    U(0)=U^0=(u^0,v^0)\nonumber.
\end{cases}
\end{equation*}
On the other hand, it is easy to verify that the adjoint operator $\mathcal{A^\ast}$, associated to $\mathcal{A}$, is defined by
\begin{equation}\label{A-adj}
 \mathcal{A^\ast}(\varphi,\psi) = \left(
\begin{array}{c}
\displaystyle \varphi_{xxx} + a_3\psi_{xxx} \\ \displaystyle
\frac{r}{b_1} \psi_{x} +\frac{1}{b_1} \psi_{xxx} +
\frac{b_2a_3}{b_1}\varphi_{xxx}
\end{array}
\right)
\end{equation}
where
$$\mathcal{A^\ast}:D(\mathcal{A^\ast})\subset X\rightarrow X$$
and
\begin{align*}D(\mathcal{A^\ast})=\left\{(\varphi,\psi)\in [H^3(0,L)]^2\,:\,\varphi(0)=\psi(0)=\varphi(L)=\psi(L)=\varphi_x(0)=\psi_x(0)=0\right\}.\end{align*}

The following results were borrowed from \cite{nina}:
\begin{proposition}\label{A-dis}
The operator $\mathcal{A}$ and its adjoint $\mathcal{A^\ast}$ are dissipative in $X$.
\end{proposition}

\begin{theorem}\label{exist-linear}Let $(u^0,v^0)\in X$. There exists a unique weak solution $(u,v)=S(\,\cdot\,)(u^0,v^0)$ of \eqref{linear} such that
\begin{equation}
(u,v)\in C([0,T];X).\nonumber
\end{equation}
Moreover, if $(u^0,v^0)\in D(\mathcal{A}),$ then \eqref{linear} has a unique (classical) solution $(u,v)$ such that
\begin{equation}
(u,v)\in C([0,T];D(\mathcal{A}))\cap C^1(0,T;X). \nonumber
\end{equation}
\end{theorem}

\begin{theorem} \label{regul-weak-linear} Let $(u^0,v^0)\in X$ and $(u,v)=S(\,\cdot\,)(u^0,v^0)$ the weak solution of \eqref{linear}. Then, $(u,v)\in L^2(0,T;[H^1(0,L)]^2)\cap H^1(0,T;[H^{-2}(0,L)]^2)$ and there exists a positive constant $c_0$ such that
\begin{equation}
||(u,v)||_{L^2(0,T;[H^1(0,L)]^2)}
\leq c_0 ||(u^0,v^0)||_{X}.\nonumber
\end{equation}
\end{theorem}

\begin{corollary}\label{exist-interp}
For any $s\in [0,3]$ and any $(u^0,v^0)\in [H^s(0,L)]^2$, the solution $(u,v)$ of \eqref{linear} belongs to $C([0,T]; [H^s(0,L)]^2)$.
\end{corollary}

\subsection{The nonlinear system} For $0 \leq s \leq 3$, let $X_s$ denote the collection of all the functions $w\in H^s(0,L)$ satisfying the s-compatibility conditions
\begin{equation*}
\begin{cases}
w(0)=w(L)=0\mbox{ when } 1/2 < s \leq 3/2, \\
w(0)=w(L)=w'(L)=0\mbox{ when } 3/2 < s \leq 3.
\end{cases}
\end{equation*}
$X_s$ is endowed with the Hilbertian norm $||w||_{H^s}$. For any $T>0$ we introduce the space
$$Y_{s,T}=C([0,T];X_{s})\cap L^{2}([0,T];H^{s+1}(0,L))$$
endowed with the norm
$$||w||_{Y_{s,T}}=||w||_{C([0,T];H^{s}(0,L))}+||w||_{L^{2}([0,T];H^{s+1}(0,L))}.$$

The next technical Lemma will be related with the nonlinear problem.

\begin{lemma}\label{tech2}
For any $T > 0$, $1 \leq p \leq 2$ and $u, v, w\in Y_{0,T}$,
\begin{eqnarray}
\int^{T}_{0}\left\|Gu\right\|_{L^{2}(\omega)}dt&\leq& CT \left\|u\right\|_{Y_{0,T}},\label{tech2.1}\\
\int^{T}_{0}\left\|uw_{x}\right\|_{L^{2}(0,L)}dt&\leq& CT^{1/4}\left\|u\right\|_{Y_{0,T}}\left\|w\right\|_{Y_{0,T}},\label{tech2.2}\\
\int^{T}_{0}\left\|u\left|w\right|^{p-1}w_{x}\right\|_{L^{2}(0,L)}dt&\leq& CT^{(2-p)/4}\left\|u\right\|_{Y_{0,T}}\left\|w\right\|^{p}_{Y_{0,T}},\label{tech2.3}\\
\int^{T}_{0}\left\|u\left|v\right|^{p-1}w_{x}\right\|_{L^{2}(0,L)}dt&\leq& CT^{(2-p)/4}\left\|u\right\|_{Y_{0,T}}\left\|w\right\|_{Y_{0,T}}\left\|v\right\|^{p-1}_{Y_{0,T}},\label{tech2.4}
\end{eqnarray}
where $C$ is a positive constant that depends only on $L$.
\end{lemma}
\begin{proof}
Estimates \eqref{tech2.2}, \eqref{tech2.3} and \eqref{tech2.4} can be obtained following closely the arguments used in \cite{Zhang}. Therefore, we will omit the proofs. 

Now, we prove the estimate \eqref{tech2.1}. By a direct computation, we have
\begin{align*}
\int^{T}_{0}||Gu||^2_{L^2(\omega)}dt&=\int^T_0\big(\int_{\omega}u^2dx-|\omega|^{-1}\big(\int_{\omega}udx\big)^2\big)^{1/2}dt\\
&\leq\int^T_0\big(\int^L_0u^2dx\big)^{1/2}dt\leq T||u||_{Y_{0,T}}.
\end{align*}
Thus \eqref{tech2.1} hold and the proof is finished.
\end{proof}

Now, we consider the following system
\begin{equation}\label{1-1a1}
\left\{
\begin{array}{l}
\vspace{1mm}u_t+u_{xxx}+a_3v_{xxx}+a(u)u_x + a_1vv_x + a_2(uv)_x + Gu =  0 \\
b_1v_t + rv_x + v_{xxx} + b_2a_3u_{xxx} + a(v)v_x + b_2a_2uu_x + b_2a_1(uv)_x + Gv = 0,\\
u(x,0)= u^0(x), \quad v(x,0)=  v^0(x)
\end{array}
\right.
\end{equation}
where $0<x<L$, $t>0$, satisfying the following boundary conditions
\begin{equation}\label{gg2a1}
\begin{cases}
u(0,t)=0,\,\,u(L,t)=0,\,\,u_{x}(L,t)=0,\\
v(0,t)=0,\,\,v(L,t)=0,\,\,v_{x}(L,t)=0,
\end{cases}
\end{equation}
where $Gu$ and $Gv$ are defined by \eqref{operator_new} and \eqref{operator_new_1}, respectively. 

The next Lemma and the well-posedness result for the system \eqref{1-1a1}-\eqref{gg2a1} were borrowed in \cite{nina}. Since the proof is similar as made in \cite{nina}, we will omit it.

\begin{lemma}\label{tech3}
Let $a=a(x)$ be a $\mathcal{C}^0$ function such that, for $0\leq p < 4$,
$$|a(x)|\leq C(1+|x|^p),\quad\forall\, x\in \mathbb{R},$$
where $C$ is a positive constant. Then, for any $T > 0$
\begin{eqnarray}
&&\left\|(u(\cdot,T),v(\cdot,T))\right\|^{2}_{X}-\left\|(u^0,v^0)\right\|^{2}_{X}\nonumber\\
&&  +\frac{1}{b_{1}}\int^{T}_{0}\left[\Big(\sqrt{b_{2}}u_{x}(0,t)+\sqrt{a^{2}_{3}b_{2}}v_{x}(0,t)\Big)^{2}
    +\left(1-a^{2}_{3}b_{2}\right)v^{2}_{x}(0,t)\right]dt\nonumber\\
&&+\frac{2}{b_{1}}\int^{T}_{0}\left(b_{2}||Gu||^{2}_{L^2(\omega)}+||Gv||^{2}_{L^2(\omega)}\right)dt=0\nonumber
\end{eqnarray}
and
\begin{equation}
\begin{array}{l}
\vspace{1mm}||(u,v)||^2_{L^2(0,T;[H^1_0(0,L)]^2)}\nonumber\\
\quad\quad\leq C\,\{(1 + T)\left\|(u^0,v^0)\right\|^{2}_{X} + T\left\|(u^0,v^0)\right\|^{6}_{X}
+T\left\|(u^0,v^0)\right\|^{\frac{8+2p}{4-p}}_{X}\}\nonumber
\end{array}
\end{equation}
where $C$ is a positive constant.
\end{lemma}

\begin{lemma}\label{existence1}
Let $a=a(x)$ be a $\mathcal{C}^1$ function such that
$$|a(x)|\leq C(1+|x|^p)\,\,\mbox{ and }\,\,|a'(x)|\leq C(1+|x|^{p-1}), \quad\forall\,x\in \mathbb{R},$$
where $C$ is a positive constant and $1\leq p < 2$. Then, for any $T > 0$ and $(u^0,v^0)\in X$ system \eqref{1-1a1}-\eqref{gg2a1} has a unique global solution.
\end{lemma}

As a consequence we get the following result.

\begin{corollary}\label{existence1-1}
Let $a=a(x)$ be a $\mathcal{C}^2$ function such that
$$|a(x)|\leq C(1+|x|^p),\,\,|a'(x)|\leq C(1+|x|^{p-1}),\,\,|a''(x)|\leq C(1+|x|^{p-2}), \quad\forall\,x\in \mathbb{R},$$ where $C$ is a positive constant and $p \geq 2$. Then, for any $(u^0,v^0)\in [H^1_0(0,L)]^2$ there exists a $T^\ast > 0$, depending only on $||(u^0,v^0)||_{[H^1(0,L)]^2}$, such that system \eqref{1-1a1}-\eqref{gg2a1} admits a unique solution $(u,v)\in L^\infty (0,T^\ast; [H^1_0(0,L)]^2)$.
\end{corollary}

\begin{proof} The ideas involved in the proof follow closely the previous arguments and those presented in \cite[Lemmas 2.11 and 2.12]{Zhang}. The extension of such results for the model under consideration was proved in \cite[Proposition 5.3]{bona} (see also \cite[Remark 5.5]{bona-zhang}). Therefore, we will omit the details.

We note that in order to apply the fixed point argument we first rewrite the system in the following integral form
\begin{align*}
U(t)&=T(t)(u^0,v^0)-\\
&\int^{t}_{0}T(t-\tau)\left(a(u)u_{x}+a_{1}vv_{x}+a_{2}(uv)_{x},\frac{a(v)}{b_{1}}v_{x}\right.\ +\left.\frac{a_{2}b_{2}}{b_{1}}uu_{x}+\frac{a_{1}b_{2}}{b_{1}}(uv)_{x} \right)d\tau,\nonumber
\end{align*}
where $\{T(t)\}_{t\geq 0}$ denotes the $C^0$ semigroup property generated by the linear part of the system. To obtain the global well-posedness one needs to establish the corresponding global a priori estimate in the space $H^1(0,L)$, which is not available.
\end{proof}

Using Lemma \ref{existence1} we prove the main result of this section:

\begin{theorem}\label{existence2}Let $a=a(x)$ be a $\mathcal{C}^2$ function such that
$$|a(x)|\leq C(1+|x|^p),\,\,|a'(x)|\leq C(1+|x|^{p-1}),\,\,|a''(x)|\leq C(1+|x|^{p-2}), \quad\forall\,x\in \mathbb{R},$$
where $C$ is a positive constant and $1\leq p < 4$. Then, for any $(u^0,v^0)\in X$, system \eqref{1-1a1}-\eqref{gg2a1} admits at least one solution $$(u,v)\in \mathcal{C}_w(\mathbb{R};X)\cap L^2_{loc}(\mathbb{R}^+;[H^1(0,L)]^2).$$
\end{theorem}
\begin{proof} We consider a sequence of functions $\left\{a_{n}\right\}_{n\in \mathbb{N}}$  in $C^{\infty}_{0}(\mathbb{R};\mathbb{R})$, such that
\begin{equation}\label{ex2-1}
\begin{cases}
\vspace{1mm}a_{n}(\mu)\rightarrow a(\mu)\quad\mbox{ uniformly on each compact set of } \mathbb{R},\\
\left|a^{j}_{n}(\mu)\right|\leq C\left(1+\left|\mu\right|^{p-j}\right), \quad\forall\,\, n\geq 0,\quad \forall\,\, \mu\in \mathbb{R},\quad j=0,1,2,
\end{cases}
\end{equation}
where $C>0$. Observe that
$$\left|a_{n}(\mu)\right|\leq C(1+\left|\mu\right|^{p})\,\,\mbox{ and }\,\,\left|a'_{n}(\mu)\right|\leq C(n)(1+\left|\mu\right|^{p-1}).$$
For each $n$, Lemma \ref{existence1} give us the existence of a unique function
\[
U=(u,v)\in \mathcal{C}(\mathbb{R}^+;X)\cap L^2(\mathbb{R};[H^1_0(0,L)]^2)
\]
which solves
\begin{equation}
\left\{
\begin{array}
[c]{l}%
\vspace{1mm}u_{n,t}+a_{n}(u_{n})u_{n,x}+u_{n,xxx}+a_{3}v_{n,xxx}+a_{1}
v_{n}v_{n,x}\\
\qquad\qquad\qquad+a_{2}(u_{n}v_{n})_{x}+Gu_{n}=0,\\
\vspace{1mm}b_{1}v_{n,t}+rv_{n,x}+a_{n}(v_{n})v_{n,x}+b_{2}a_{3}
u_{n,xxx}+v_{n,xxx}\\
\vspace{1mm}\qquad\qquad\qquad+b_{2}a_{2}u_{n}u_{n,x}+b_{2}a_{1}(u_{n}
v_{n})_{x}+Gv_{n}=0,\\
\vspace{1mm}u_{n}(0,t)=u_{n}(L,t)=u_{n,x}(L,t)=0,\\
\vspace{1mm}v_{n}(0,t)=v_{n}(L,t)=v_{n,x}(L,t)=0,\\
u_{n}(x,0)=u_n^{0}(x),\,\,\,v_{n}(x,0)=v_n^{0}(x),
\end{array}
\right.  \label{ex2-2}
\end{equation}
where $0<x<L$ and $t>0$. Moreover, from Lemma \ref{tech3}, we get
\begin{equation}
\begin{array}{l}
\vspace{1mm}\left\|(u_{n}(\cdot,T),v_{n}(\cdot,T))\right\|^{2}_{X}+\displaystyle\frac{1}{b_{1}}\int^{T}_{0}[(\sqrt{b_{2}}u_{n,x}(0,t)+\sqrt{a^{2}_{3}b_{2}}v_{n,x}(0,t))^{2}
+\left(1-a^{2}_{3}b_{2}\right)v^{2}_{n,x}(0,t)]dt\nonumber\\
    +\displaystyle\frac{2}{b_{1}}\int^{T}_{0}\left(b_{2}||Gu_{n}||^{2}_{L^2(\omega)}+||Gv_{n}||^{2}_{L^2(\omega)}\right)dxdt=\left\|(u_n^0,v_n^0)\right\|^{2}_{X}\nonumber
\end{array}
\end{equation}
and
\begin{equation*}
\displaystyle\int^{T}_{0}\int^{L}_{0}u^{2}_{n,x}dxdt+\int^{T}_{0}\int^{L}_{0}v^{2}_{n,x}dxdt\leq C\,\{(1 + T)\left\|(u_n^{0},v_n^{0})\right\|^{2}_{X} + T\left\|(u_n^{0},v_n^{0})\right\|^{6}_{X}
+T\left\|(u_n^{0},v_n^{0})\right\|^{\frac{8+2p}{4-p}}_{X}\},
\end{equation*}
for any $T\geq 0$, where $C>0$. Due the estimates above, we have that the sequence $$\{U_n\}_{n\in \mathbb{N}}=\{(u_n,v_n)\}_{n\in \mathbb{N}} \text{ is bounded in } L^\infty(\mathbb{R}^+;X)\cap L^2_{loc}(\mathbb{R}^+;[H^1_0(0,L)]^2).$$ Therefore, there exists a function $U=(u,v)$ and a subsequence, still denoted by the same index $n$, such that
\begin{eqnarray}
&&U_{n}\rightarrow U\quad\mbox{ weakly }^\ast\quad\mbox{ in } \quad L^{\infty}(\mathbb{R}^+;X)\label{ex2-3}\\
&&U_{n}\rightarrow U\quad\mbox{ weakly }\quad\mbox{ in } \quad L^{2}_{loc}(\mathbb{R}^+; [H^{1}_{0}(0,L)]^2)\label{ex2-4}.
\end{eqnarray}

The goal now is to pass the limit in \eqref{ex2-2} to prove that $U=(u,v)$ is a weak solution of the problem \eqref{1-1a1}-\eqref{gg2a1}. Here, the main difficult is the study of the nonlinear terms. In order to deal with this, we introduce the following functions
\begin{eqnarray}
    A_{n}(\varphi):=\int^{\varphi}_{0}a_{n}(s)ds \quad\mbox{ and }\quad \widetilde{A}_{n}(\varphi):=\int^{\varphi}_{0}sa_{n}(s)ds
\end{eqnarray}
and we will prove that
\begin{eqnarray}
    (A_{n}(u_{n}),A_{n}(v_{n})) \rightarrow (A(u),A(v))\quad\mbox{in}\quad [\mathcal{D}'((0,L)\times (0,+\infty))]^2, \quad\mbox{ as } n\rightarrow\infty.
\end{eqnarray}
To obtain this result, we divide the prove in several steps:

\vglue 0.2 cm

\noindent {\sc Step 1.} For any $T >0$ and $\alpha\in (1,\frac{6}{p+1}]$, the sequence $\{(A_{n}(u_{n}),A_{n}(v_{n}))\}_{n\in \mathbb{N}}$ is bounded in the space $[L^\alpha((0,T)\times (0,L))]^2$.

\vglue 0.1 cm

Indeed, from \eqref{ex2-1}
\begin{eqnarray}
\left|A_{n}(\varphi)\right|\leq \int^{\varphi}_{0}\left|a_{n}(s)\right|ds\leq \int^{\varphi}_{0}C_1(1+\left|s\right|^{p})dv\leq C_1(1+\left|\varphi\right|^{p+1}),\nonumber
\end{eqnarray}
which gives that
\begin{eqnarray}
    \left|A_{n}(\varphi)\right|^{\alpha}\leq C_2(1+\left|\varphi\right|^{\alpha(p+1)}),\nonumber
\end{eqnarray}
for some positive constants $C_1$ and $C_2$ that does not depend on $n$. Then, since $$\frac{\alpha(p+1)-2}{2}\leq 2,$$ we can combine Lemma \ref{tech3} and Gagliardo-Nirenberg's inequality to obtain
\begin{align*}
\int^{T}_{0}\int^{L}_{0}\left|A_{n}(u_{n})\right|^{\alpha}dxdt&\leq C_2\left(TL+\int^{T}_{0}\int^{L}_{0}\left|u_{n}\right|^{\alpha(p+1)}dxdt\right)\\
&\leq C_2\left(TL+\vspace{1mm}\int^{T}_{0}\left\|u_{n}\right\|^{\alpha(p+1)-2}_{L^{\infty}(0,L)}\left\|u_{n}\right\|^{2}_{L^{2}(0,L)}dt\right)\\
&\leq C_2TL+C2C_3\int^{T}_{0}\left\|u_{n}\right\|^{\frac{\alpha(p+1)-2}{2}}_{L^{2}(0,L)}\left\|u_{n,x}\right\|^{\frac{\alpha(p+1)-2}{2}}_{L^{2}(0,L)}\left\|u_{n}\right\|^{2}_{L^{2}(0,L)}dt\\
&\leq C_2\left(TL+C_3\left\|u^0\right\|^{\frac{\alpha(p+1)}{2}+1}\displaystyle\int^{T}_{0}\left\|u_{n,x}\right\|^{\frac{\alpha(p+1)-2}{2}}_{L^{2}(0,L)}dt
\right)\\
&\leq C_4,
\end{align*}
for some $C_3 > 0$ and $C_4=C_4(T,||(u^0,v^0)||_X) > 0$. Analogously,
\begin{eqnarray}
\int^{T}_{0}\int^{L}_{0}\left|A_{n}(v_{n})\right|^{\alpha}dxdt&\leq& C_4,\nonumber
\end{eqnarray}
which complete the proof of the Step 1.

\vglue 0.2 cm

\noindent {\sc Step 2.}  For any $T>0$ and $\alpha\in (1,\frac{6}{p+1}]$, the sequence  $\{U_{n,t}\}_{n\in \mathbb{N}}=\{(u_{n,t},v_{n,t})\}_{n\in \mathbb{N}}$
is bounded in $L^\alpha(0,T;[H^{-2}(0,L)]^2)$.

\vglue 0.1 cm

The estimates obtained for $\{U_n\}_{n\in \mathbb{N}}$ guarantees that the terms $v_nv_{n,x}, (u_nv_n)_x$, $u_nu_{n,x}$ and $(u_nv_n)_x$ that appears in \eqref{ex2-2} are bounded in $L^\alpha(0,T;[H^{-2}(0,L)]^2 )$. In fact, observe that $\alpha\leq 2$, $L^1(0,L)\subset H^{-2}(0,L)$ and $$\|u_n v_{n,x}\|_{L^2(0,T;L^1(0,L))}\leq \|u_n\|_{L^\infty(0,T;L^2(0,L))}\|v_n\|_{L^2(0,T;H^1_0(0,L))}.$$ The same result is valid for the linear terms. On the other hand, due to the embedding $$L^{\alpha}(0,L)\hookrightarrow H^{-1}(0,L)$$ and by using Step 1 we conclude that $$\partial_x(A_n(u_n),A_n(v_n)) = (a(u_n)u_{n,x},a(v_n)v_{n,x}) \mbox{ is bounded in } [L^2(0,T;[H^{-2}(0,L)]^2)]^2.$$ 
Now, observing that
\begin{align*}
u_{n,t}&=-(a(u_{n})u_{n,x}+u_{n,xxx}+a_{3}v_{n,xxx}+a_{1}v_{n}v_{n,x}+a_{2}(u_{n}v_{n})_{x}+Gu_{n})
\end{align*}
and
\begin{align*}
b_{1}v_{n,t}&=-(rv_{n,x}+a(v_{n})v_{n,x}+b_{2}a_{3}u_{n,xxx}+v_{n,xxx}+b_{2}a_{2}u_{n}u_{n,x})-(b_{2}a_{1}(u_{n}v_{n})_{x}+Gv_{n})
\end{align*}
the result holds.

\vglue 0.2 cm

\noindent {\sc Step 3. (Ergoroff Theorem)} {\it Let $\Omega$ be an open subset in $\mathbb{R}^N$. If $\{f_n\}_{n\in \mathbb{N}}$ is a sequence in $L^p(\Omega)$, with $1 < p <\infty$, such that $f_n\rightharpoonup f$ and $f_n(x)\rightarrow g(x)$ a.e., as $n\rightarrow\infty$, then $f(x)=g(x)$ a.e.}

\vglue 0.2 cm

\noindent Now, by the Steps 1--3,  we can complete the proof. Indeed, since 
\[
\{U_n\}_{n\in \mathbb{N}} \text{ is bounded in } L^2(0,T;[H^1_0(0,L)]^2)
\]
and 
\[
\{U_{n,t}\}_{n\in \mathbb{N}} \text{ in } L^2(0,T;[H^{-2}(0,L)]^2),
\]
from \cite[Corollary 4]{simon}, we obtain a subsequence, still denoted by the same index, such that
\begin{equation}\label{ex2-5}
U_n\rightarrow U\quad\mbox{ in }\quad [L^2(0,T;L^2(0,L))]^2,\quad\mbox{ strongly and a. e.}
\end{equation}
Then, from \eqref{ex2-1} and \eqref{ex2-5}, we have
\begin{eqnarray}
(A_{n}(u_{n}(x,t)),A_{n}(v_{n}(x,t)))\rightarrow (A(u(x,t)),A(v(x,t)))\text{ a.e. for } (x,t)\in (0,L)\times\mathbb{R}^+.\nonumber
\end{eqnarray}
Moreover, from Step 1, we can pass to a subsequence, if necessary, to obtain a function $$g=(g_1,g_2) \in L^\alpha(0,T; [L^\alpha(0,L)]^2)$$ for which
$$(A_{n}(u_{n}(x,t)),A_{n}(v_{n}(x,t)))\rightharpoonup (g_1,g_2)\quad\mbox{ weakly in } [L^\alpha(0,T; L^\alpha(0,L))]^2.$$
Consequently, Ergoroff Theorem (see Step 3) allows us to conclude that 
$$(A(u(x,t)),A(v(x,t)))=(g_1,g_2)$$ and then
$$(A_{n}(u_{n}(x,t)),A_{n}(v_{n}(x,t)))\rightarrow (A(u(x,t)),A(v(x,t)))\text{ in }[\mathcal{D}'((0,L)\times (0,+\infty))]^2.$$
Taking the spatial derivative we deduce that
$$(a_n(u_n)u_{n,x},a_n(v_n)v_{n,x})\rightarrow (a(u)u_x,a(v)v_x)\text{ in }[\mathcal{D}'((0,L)\times (0,+\infty))]^2.$$
Finally, putting the convergences above together we can pass the weak limit in the system \eqref{ex2-2}. However, to conclude that $U$ is a weak solution of \eqref{ex2-2} it remains to prove that $U$ satisfies $$U(x,0)=(u^0(x),v^0(x))$$ and $$U\in \mathcal{C}_w([0,T]; X).$$ As $\{U_n\}_{n\in \mathbb{N}}$ is bounded in $L^\infty(0,T; X)$
and $$\{U_{n,t}\}_{n\in \mathbb{N}} \in L^\alpha(0,T;[H^{-2}(0,L)]^2),$$ with $\alpha > 1$, we can apply again \cite[Corollary 4]{simon} to obtain a subsequence $\{U_n\}_{n\in \mathbb{N}}$ satisfying
\begin{equation}\label{ex2-6}
U_n\rightarrow U\quad\mbox{ in }\quad \mathcal{C}([0,T];[H^{-1}(0,L)]^2),\quad\mbox{ for any } T>0.
\end{equation}
In particular,
$$U^0_n(x)=U_n(x,0)\rightarrow U(x,0),$$
since $$U\in L^\infty(0,T; X)\cap \mathcal{C}([0,T]; [H^{-1}(0,L)]^2),$$ from \cite[Lemma 1.4]{teman} we deduce that $U\in\mathcal{C}_w([0,T]; X)$. Therefore, the prove of Theorem \ref{existence2} is archived.
\end{proof}

\section{Exponential stabilization\label{sec3}}

In this section we prove the uniform exponential decay of the total energy $E_s(t)$, defined by \eqref{6a}, associated to the following system 
\begin{equation}\label{1-1a3}
\left\{
\begin{array}{l}
\vspace{1mm}u_t+u_{xxx}+a_3v_{xxx}+a(u)u_x + a_1vv_x + a_2(uv)_x + Gu =  0 \\
b_1v_t + rv_x + v_{xxx} + b_2a_3u_{xxx} + a(v)v_x + b_2a_2uu_x + b_2a_1(uv)_x + Gv = 0,\\
u(0,t)=0,\,\,u(L,t)=0,\,\,u_{x}(L,t)=0,\\
v(0,t)=0,\,\,v(L,t)=0,\,\,v_{x}(L,t)=0,\\
u(x,0)= u^0(x), \quad v(x,0)=  v^0(x)
\end{array}
\right.
\end{equation}
where $0<x<L$, $t>0$, with
\begin{equation}\label{operator_new1}
Gu=1_{\omega}\left(u(t,x)-\frac{1}{|\omega|}\int_{\omega}u(t,x)dx\right)
\end{equation}
and
\begin{equation}\label{operator_new_11}
Gv=1_{\omega}\left(v(t,x)-\frac{1}{|\omega|}\int_{\omega}v(t,x)dx\right),
\end{equation}
where $\omega\subset(0,L)$ is a nomempty open set and $1_{\omega}$ denotes the characteristic function of the set $\omega$. 

To show our main result we will use the so-called "Compactness-Uniqueness Argument" due to J.-L. Lions (see \cite{Lions}) which reduces the problem to prove a Unique Continuation Property for weak solutions. As the weak solution of \eqref{1-1a3} may fail to be unique, we will say that the solution is exponential stable if the following property holds.

\begin{definition}\label{def}
System \eqref{1-1a3} is said to be locally uniformly exponentially stable in $X$ if for any $R > 0$ there exist positive constants $C$ and $\alpha$ such that for any $U^0=(u^0,v^0)$ with $E(0) \leq R$ and for any weak solution $U=(u,v)$ of \eqref{1-1a3}, the following holds
$$E_s(t) \leq C\, E_s(0) e^{-\alpha t},\quad \forall\,\,t > 0.$$ If the constant $\alpha$ is independent of $R$, the system \eqref{1-1a3} is said to be globally uniformly exponentially stable in $X$.
\end{definition}

The next proposition give us a local uniform result.

\begin{proposition}\label{local-dec}
Let $a=a(x)$ be a $\mathcal{C}^2$ function such that $$|a(x)|\leq C(1+|x|^p),\,\,|a'(x)|\leq C(1+|x|^{p-1}),\,\,|a''(x)|\leq C(1+|x|^{p-2}), \quad\forall\,x\in \mathbb{R}$$ where $C$ is a positive constant and $1\leq p < 4$. Then, the system \eqref{1-1a3} is locally uniformly stable.
\end{proposition}

\begin{proof} To obtain the exponential decay of $E_s(t)$ is known be necessary to prove the following \textit{observability inequality} 
\begin{equation}\label{dec1}
\begin{array}{l}
\vspace{1mm} E_s(0)\leq C \displaystyle\int_0^T\Big(b_2||Gu||^2_{L^2(\omega)}+ ||Gv||^2_{L^2(\omega)}\Big)\\
+\displaystyle\int_0^T\Big[\frac{1}{2}\left(\sqrt{b_{2}}u_{x}(0,t)+\sqrt{a^{2}_{3}b_{2}}v_{x}(0,t)\right)^{2}
\displaystyle+\frac{1}{2}\left(1-a^{2}_{3}b_{2}\right)v^{2}_{x}(0,t)\Big]dt,
\end{array}
\end{equation}
for every finite energy solution of \eqref{1-1a3}, where $C=C(R,T)$ is a positive constant. To prove \eqref{dec1} we first multiply the first equation of \eqref{1-1a3} by $(T-t)b_2 u$ and add with the second one multiplied by $(T-t)v$. Therefore, by using integration by parts, we have
\begin{equation}\label{dec2}
\begin{array}{l}
\displaystyle\frac{b_{2}}{2}\int^{L}_{0}(u^{0})^{2}dx+\frac{b_{1}}{2}\int^{L}_{0}(v^{0})^{2}dx\leq \displaystyle\frac{1}{T}\left( \frac{b_{2}}{2}\int^{T}_{0}\int^{L}_{0}u^{2}dxdt+\frac{b_{1}}{2}\int^{T}_{0}\int^{L}_{0}v^{2}dxdt\right)\\
\vspace{1mm}+\displaystyle\frac{1}{2}\int^{T}_{0}\Big[(\sqrt{b_{2}}u_{x}(0,t)+\sqrt{a^{2}_{3}b_{2}}v_{x}(0,t))^{2}
    +\left(1-a^{2}_{3}b_{2}\right)v^{2}_{x}(0,t)\Big]dt\\
    +\displaystyle\int^{T}_{0}(b_{2}||Gu||^2_{L^2(\omega)}+||Gv||^2_{L^2(\omega)})dt.
\end{array}
\end{equation}
Thus, to obtain \eqref{dec1} we have to prove the following claim:

\vglue 0.2 cm

\textit{For any $T>0$ and $R>0$, there exists a constant $C(R,T)>0$ satisfying
\begin{equation}
\begin{array}{lll}\label{dec3}
\vspace{1mm}\displaystyle\frac{b_{2}}{b_1}\int^{T}_{0}\int^{L}_{0}u^{2}dxdt+\int^{T}_{0}\int^{L}_{0}v^{2}dxdt\\
\vspace{1mm}\displaystyle\leq C(R,T)\Big(\int^{T}_{0}\Big[(\sqrt{b_{2}}u_{x}(0,t)+\sqrt{a^{2}_{3}b_{2}}v_{x}(0,t))^{2}
    +\left(1-a^{2}_{3}b_{2}\right)v^{2}_{x}(0,t)\Big]dt\\
+ \displaystyle\int^{T}_{0}2(b_2||Gu||^2_{L^2(\omega)}+||Gv||^2_{L^2(\omega)})dt\Big)
\end{array}
\end{equation}
for any weak solution $U$ of \eqref{1-1a3}, whenever $||(u^0,v^0)||_X \leq R$.}

\vglue 0.1 cm

Indeed, if \eqref{dec3} is not true, there exists a sequence of functions $$\{U_{n}\}_{n \in\mathbb{N}}=\{(u_{n},v_{n})\}_{n \in \mathbb{N}}\in C_w([0,T];X)\cap L^2(0,T; [H^1_0(0,L)]^2),$$ such that
\begin{eqnarray}\label{dec3.1}
    ||(u_{n}(\cdot,0),v_{n}(\cdot,0))||_X \leq R,
\end{eqnarray}
solution of
\begin{equation}\label{dec4}
\left\{
\begin{array}{l}
\vspace{2mm}
b_1U_{n,t}+AU_{n,xxx}+RU_{n,x}+B(U_{n})U_{n,x}+C(U_n)U_n + \mathcal{G}U_{n}=0,
\\ \vspace{2mm}U_{n}(0,t)= U_{n}(L,t)= U_{n,x}(L,t)=0, \\
U_{n}(x,0)=U_n^0(x),
\end{array}
\right.
\end{equation}
where $x\in(0,L),\,t>0$ with $U=(u,v)$,
\[
A=\left(\begin{array}[c]{cc}b_1 & b_1a_3\\b_2a_3 & 1 \end{array}\right),
\]
\[
B(U)=\left(\begin{array}[c]{cc}b_1a_2v & b_1(a_2u+a_1v)\\ b_2(a_2u+a_1v) & b_2a_1u \end{array}\right)
\]
\[
R=\left(\begin{array}[c]{cc}0 & 0\\ 0 & r\end{array}\right)
\]
and
\[
C(U)=\left(\begin{array}[c]{cc}b_1a(u) & 0\\0 & a(v)\end{array}\right).
\]
$\mathcal{G}$ is as diagonal matrix whose diagonal elements are damping operators $b_2G$ and $G$. Additionally,
\begin{equation}\label{dec5}
\lim_{n\rightarrow \infty}\frac{\frac{b_2}{b_1}||u_n||_{L^2(0,T;
L^2(0,L))}^2+||v_n||_{L^2(0,T; L^2(0,L))}^2}{I_n}=\infty,
\end{equation}
where
\begin{equation}
\begin{array}{l}
\vspace{1mm} I_n=\displaystyle\int_0^T
\left[(\sqrt{b_{2}}u_{n,x}(0,t)+\sqrt{a^{2}_{3}b_{2}}v_{n,x}(0,t))^{2}
    +\left(1-a^{2}_{3}b_{2}\right)v^{2}_{n,x}(0,t) \right] dt\\
    + 2\displaystyle\int^{T}_{0}(b_2||Gu||^2_{L^2(\omega)}+||Gv||^2_{L^2(\omega)})dt.\nonumber
\end{array}
\end{equation}
Let
\begin{equation}\label{dec6}
\sigma_{n}:=\left( \frac{b_2}{b_1}||u_n||_{L^2(0,T; L^2(0,L))}^2+||v_n||_{L^2(0,T;
L^2(0,L))}^{2}\right)^{\frac{1}{2}}=||U_{n}||_{L^2(0,T;X)}
\end{equation}
and consider
\begin{equation}\label{dec7}
Z_{n}(x,t):=\frac{1}{\sigma_{n}}U_{n}(x,t)=
\left(
\begin{array}{c}
y_{n}(x,t)\\ w_{n}(x,t)
\end{array}
\right).
\end{equation}
For each $n\in \mathbb{N}, \, Z_{n}$ satisfies the following system
\begin{equation}\label{dec8}
\left\{
\begin{array}{l}
\vspace{2mm} b_1Z_{n,t}+AZ_{n,xxx}+RZ_{n,x}+\sigma_nB(Z_{n})Z_{n,x}+ C(\sigma_nZ_n)Z_n + \mathcal{G}Z_{n}=0,
\\ \vspace{2mm}Z_{n}(0,t)=Z_{n}(L,t)=Z_{n,x}(L,t)=0, \\
Z_{n}(x,0)=Z_n^0(x)=\frac{U_n^0(x)}{\sigma_n},
\end{array}
\right.
\end{equation}
with $0<x<L$ and $t>0$,
\begin{equation}\label{dec9}
||Z_{n}||_{L^{2}(0,T; X) }^{2}=1
\end{equation}
and
\begin{equation}\label{dec10}
\begin{array}{l}
\vspace{1mm}\displaystyle\int_0^T
\displaystyle[(\sqrt{b_{2}}y_{n,x}(0,t)+\sqrt{a^{2}_{3}b_{2}}w_{n,x}(0,t))^{2}+\left(1-a^{2}_{3}b_{2}\right)w_{n,x}^{2}(0,t)]dt\\
     + \displaystyle\int_0^T2(b_2||Gy_n||^2_{L^2(\omega)} + ||Gw_n||^2_{L^2(\omega)})dt\rightarrow 0,
\end{array}
\end{equation}
as $n\rightarrow\infty$. Observe that the energy dissipation law and \eqref{dec3.1} guarantee that $\sigma_n$ is bounded. Then, extracting a subsequence, still denoted by the same index, we can assume that
$$\sigma_n\rightarrow \sigma  \geq 0.$$
Moreover, combining \eqref{dec2}, \eqref{dec9}  and \eqref{dec10} we deduce that $||Z_n^0||_X$ is bounded. Then, following the arguments used in the proof of Theorem \ref{existence2}, there exists a function $Z=(y,w)$ such that
\begin{equation}
\begin{array}{l}\label{dec11}
\vspace{2mm}Z_{n}\rightharpoonup Z\quad\mbox{ in } \quad L^{\infty}(0,T;X)\quad \mbox{ weak * },\\
\vspace{2mm}Z_{n}\rightharpoonup Z\quad\mbox{ in }\quad L^{2}(0,T;[H^{1}(0,L)]^2)\quad \mbox{ weak },\\
\vspace{2mm}Z_{n}\rightarrow Z\quad\mbox{ in } \quad L^{2}(0,T;X)\quad \mbox{ a. e. },\\
\vspace{2mm}Z_{n}\rightarrow Z\quad\mbox{ in } \quad C([0,T];[H^{-1}(0,L)]^2),\\
    C(\sigma_{n}Z_{n})Z_{n,x}\rightarrow C(\sigma Z)Z_{x}\quad\mbox{ in } \quad
D'((0,L)\times(0,T)).
\end{array}
\end{equation}
The last convergence follows from the fact that
\begin{eqnarray}
    \left|a(\sigma_{n}\mu)\right|\leq C(1+\left|\sigma_{n}\right|^{p}\left|\mu\right|^{p})\leq C'(1+\left|\mu\right|^{p}),\nonumber
\end{eqnarray}
where $C'$ is a positive constant. Consequently, by \eqref{dec10} and \eqref{dec11} we obtain
\begin{equation}\label{dec12}
||Z||_{L^2(0,T;X)}=1
\end{equation}
and
\begin{equation}\label{dec13}
\begin{array}{l}
\vspace{1mm}\displaystyle\int_0^T
\left[(\sqrt{b_{2}}y_{x}(0,t)+\sqrt{a^{2}_{3}b_{2}}w_{x}(0,t))^{2}
    +\left(1-a^{2}_{3}b_{2}\right)w^{2}_{x}(0,t) \right]dt\\
     + \displaystyle\int_0^T2(b_2||Gy||^2_{L^2(\omega)}+||Gw||^2_{L^2(\omega)})dt \leq 0.
\end{array}
\end{equation}
The previous statements ensures that $Z$ fulfills
\begin{equation}\label{dec14}
\left\{
\begin{array}{l}
\vspace{2mm} b_1Z_t+AZ_{xxx}+RZ_x+\sigma B(Z)Z_x + C(\sigma Z)Z+ \mathcal{G}Z=0,\\ Z(0,t)=Z(L,t)=0,
\end{array}
\right.
\end{equation}
in $D'(\omega\times(0,T))$. In addition, from \eqref{dec13}, we get $$Gy=0$$ and $$Gw=0,$$
or equivalently,
$$y(x,t)-\frac{1}{|\omega|}\int_{\omega}y(x,t)dx=0 \iff y(x,t)=\frac{1}{|\omega|}\int_{\omega}y(x,t)dx$$
and
$$w(x,t)-\frac{1}{|\omega|}\int_{\omega}w(x,t)dx=0 \iff w(x,t)=\frac{1}{|\omega|}\int_{\omega}w(x,t)dx,$$
which implies that 
$$y(x,t)=s_1(t) \text{ on } \omega\times(0,T)$$
and
$$w(x,t)=s_2(t) \text{ on } \omega\times(0,T),$$
for some functions $s_1(t)$ and $s_2(t)$. Therefore, we have that $Z$ fulfills
\begin{equation}\label{dec14a}
\left\{
\begin{array}{l}
\vspace{2mm} b_1Z_t+AZ_{xxx}+RZ_x+\sigma B(Z)Z_x + C(\sigma Z)Z+ \mathcal{G}Z=0,\\ Z(0,t)=Z(L,t)=0,\\ Z=\left(
\begin{array}{c}
s_1(t)\\ s_2(t)
\end{array}
\right) \text{ on } \omega\times(0,T). \\
\end{array}
\right.
\end{equation}
The first equation of \eqref{dec14a} gives $Z^{\prime}(x,t)=0$ implying $s^{\prime}_1=s^{\prime}_2=0$ which, combined with a unique continuation property proved in \cite[Corollary 3]{nina}, yields that $Z(x,t)\equiv s$ for some constant $s\in\mathbb{R}$ in $(0,L)\times(0,T)$. Since $Z(L,t)=0$, we deduce that
\begin{equation}\label{dec14-1}
Z\equiv 0\quad\mbox{ on }\quad (0,L)\times(0,T).
\end{equation}
Therefore, this completes the proof.
\end{proof}

\begin{remark}
Observe that to apply \cite[Corollary 3]{nina} we need $Z\in L^\infty(0,T;[H^1(0,L)]^2)$, however, by using \cite[Proposition 3]{nina} with minor changes, we have $Z$ in the appropriate class.
\end{remark}

Now we are able to prove the main result of this paper.

\begin{proof}[Proof of Theorem \ref{main-dec}]
Proposition \ref{local-dec} guarantees the existence of a constant $\alpha > 0$, such that if $E_s(0) < 1$, the corresponding solution fulfill
$$E_s(t) \leq C' E_s(0) e^{-\alpha t},\quad \forall\, t> 0,$$
where $E_s(t)$ is defined by \eqref{6a}. Moreover, given $R > 0$ we obtain positive constants $C=C(R)$ and $\beta=\beta(R)$ such that
$$E_s(t) \leq C E_s(0) e^{-\beta t},\quad \forall\, t> 0,$$
whenever $E_s(0) < R$. Then, setting $T_R:=\beta^{-1}\ln(RC)$, we get
$$E_s(t) \leq C' E_s(T_R) e^{-\alpha(t-T_R)},\quad \forall\, t>T_R,$$
which ensures that
$$E_s(t) \leq C' C E_s(0) e^{\alpha T_R} e^{-\alpha t},\quad \forall\, t> 0.$$
This completes the proof and main Theorem is proved.
\end{proof}

\section{Extension Results: Stabilization for the critical case and weak solutions\label{sec4}}

\subsection{The critical case}

In this subsection we will follow the arguments due to \cite{linpa} to prove that for the critical case $a(u)=u^4$ we have the exponential decay of the total energy $E_s(t)$ in space $X=[L^2(0,L)]^2$, assuming that $||(u^0,v^0)||_X < < 1$. Moreover,  for $a(u)=u^p$, $p\in(2,4)$, the existence of weak solutions is also verified.

\subsubsection{Exponential decay}

Note that the energy dissipation law \eqref{dissipation}, as well as, \eqref{dec2} is still valid when $a(u)=u^4$. Then, the main result of this subsection can be read as follows:

\begin{theorem}\label{main-dec1} Consider $E_s(t)$ defined by \eqref{6a}. Then, there exist positive constants $C$ and $k$, such that for any
$(u^0,v^0)\in [L^2(0,L)]^2$ with
\[
E_s(0)\leq R_0,
\]
the corresponding solution $(u,v)$ of
\begin{equation}\label{1-1aa}
\left\{
\begin{array}{l}
\vspace{1mm}u_t+u_{xxx}+a_3v_{xxx}+u^4u_x + a_1vv_x + a_2(uv)_x + Gu =  0 \\
b_1v_t + rv_x + v_{xxx} + b_2a_3u_{xxx} + v^4v_x + b_2a_2uu_x + b_2a_1(uv)_x + Gv = 0,
\end{array}
\right.
\end{equation}
satisfying the following boundary conditions
\begin{equation}\label{gg2aa}
\begin{cases}
u(0,t)=0,\,\,u(L,t)=0,\,\,u_{x}(L,t)=0,\\
v(0,t)=0,\,\,v(L,t)=0,\,\,v_{x}(L,t)=0,
\end{cases}
\end{equation}
where $0<x<L$, $t>0$, $Gu$ and $Gv$ are defined by \eqref{operator_new} and \eqref{operator_new_1}, respectively, satisfies
\begin{equation}
E_s(t)\leq Ce^{-kt}E_s(0)\text{, }\quad \forall t\geq 0.
\label{m31}
\end{equation}
\end{theorem} 

\begin{proof}
To prove the exponential decay the following claim will be needed.

\vglue 0.2 cm

\noindent {\sc Claim.} \textit{For any $T>0$ and $R>0$, there exists a constant $C=C(R,T)>0$, such that
\begin{equation}
\begin{array}{l}\label{critical-2}
\vspace{1mm}\displaystyle\int^{T}_{0}\int^{L}_{0}u^{2}dxdt+\frac{b_{1}}{b_2}\int^{T}_{0}\int^{L}_{0}v^{2}dxdt\\
\vspace{1mm}\displaystyle\leq C(R,T)\Big(\displaystyle\int^{T}_{0}[(\sqrt{b_{2}}u_{x}(0,t)+\sqrt{a^{2}_{3}b_{2}}v_{x}(0,t))^{2}
    +\left(1-a^{2}_{3}b_{2}\right)v^{2}_{x}(0,t)]dt\\
+\displaystyle\int^{T}_{0}2(b_2||Gu||^2_{L^2(\omega)}+||Gv||^2_{L^2(\omega)})dt\Big),
\end{array}
\end{equation}
for any solution solution of \eqref{1-1aa}-\eqref{gg2aa}, whenever $||(u^0,v^0)||^2_{X} \leq R^2$.} 

\vglue 0.2 cm

To prove the previous claim we use the approach developed in the proof of the Proposition \ref{local-dec}. To use it, the following estimates are required.

\vglue 0.1 cm

\noindent{\bf Estimate I.} Multiplying the first equation in \eqref{1-1aa} by $xu$, the second by $xv$ and integrating by parts we obtain
\begin{equation}\label{critical-3}
\begin{array}{c}
\displaystyle\int_0^T\int_0^L (u^2_x + v^2_x) dx dt \leq
C\left[||(u^0,v^0)||_X^2 + \int_0^T\int_0^L (u^6 + v^6) dx dt
\right],
\end{array}
\end{equation}
where $C=C(T,L)$ is a positive. Now, we will bound $(u,v)$ in $L^2(0,T;[H^1_0(0,L)]^2$. By using the Gagliardo-Nirenberg inequality and \eqref{dissipation}, we get that
\begin{equation}\label{critical-4}
\begin{array}{l}
\vspace{1mm}\displaystyle\int_0^T\int_0^L u^6 dx dt \leq
C\,\displaystyle\int_0^T \|u(t)\|_{L^2(0,L)}^4
\|u_x(t)\|_{L^2(0,L)}^2 dt\\
\qquad\qquad\qquad\quad\leq
 C\,\|(u^0,v^0)\|^4_X\displaystyle\int_0^T \|u_x(t)\|^2_{L^2(0,L)} dt
\end{array}
\end{equation}
for some constant $C>0$. Analogously, we can estimate $\int_0^T\int_0^Lv^6 dx dt$. Thus, by \eqref{critical-3}, we have that
\begin{equation}\label{critical-5}
\begin{array}{c}
(1-C||(u^0,v^0)||^4_X)||(u,v)||^2_{L^2(0,T;[H^1_0(0,L)]^2)} \leq
C(T,L)||(u^0,v^0)||_X^2.
\end{array}
\end{equation}

\vglue 0.1 cm

\noindent{\bf Estimate II.} Now, we need bound $u_t$, respectively $v_t$. To do this, we we should have to pay some attention to the nonlinear term $u^4u_x=\displaystyle\frac{1}{5}\partial_x(u^5)$, respectively $v^4v_x$. First, observe that the argument used in \eqref{critical-4} give us
\begin{equation*}
\begin{array}{l}
\vspace{1mm}\displaystyle\int_0^T\int_0^L|u^5|^{6/5} dx dt\leq c\,\|u^0\|^4 \,\int_0^T
\|u_x(t)\|^2_{L^2(0,L)}dt\leq c\,\|(u^0,v^0)\|^4_X \,\int_0^T
\|u_x(t)\|^2_{L^2(0,L)}dt.
\end{array}
\end{equation*}
Therefore, from the \eqref{dissipation} and \eqref{critical-5}, the following holds $$\{u^5\}\,\, \mbox{is bounded in}\,\,L^{\frac{6}{5}}((0,T)\times (0,L)).$$ On the other hand, since $$L^{\frac{6}{5}}(0,L)\hookrightarrow H^{-1}(0,L),$$ we conclude that $$\{u^4u_x\}=\{\displaystyle\frac{1}{5}\partial_x(u^5)\}\,\, \mbox{is bounded in} \,\,L^{\frac{6}{5}}(0,T;H^{-2}(0,L)).$$ The result is also verified for $v^4v_x$.

\vglue 0.1 cm

\noindent{\bf Estimate III.}  Now, we can obtain a bound for $(u_t,v_t)$. Indeed, since
\begin{equation}
\left\{
\begin{array}{l}
\vspace{1mm}u_t= -(u_{xxx} +  a_3v_{xxx} + u^4u_x + a_1vv_x +
a_2(uv)_x + Gu), \nonumber\\
b_1v_t = -( v_x + v_{xxx} + b_2a_3u_{xxx} + v^4v_x + b_2a_2uu_x +
b_2a_1(uv)_x + Gv),\nonumber
\end{array}
\right.
\end{equation}
the second estimate allows to conclude that
\begin{equation}\label{critical-6}
(u_t,v_t)\,\, \mbox{is bounded
in}\,\,L^{\frac{6}{5}}(0,T;[H^{-2}(0,L)]^2).
\end{equation}

\vglue 0.1 cm

\noindent{\bf Estimate IV.} Finally, arguing by contradiction (see Proposition \ref{local-dec}), the main difficult is to pass to the limit in the nonlinear term $w_n^4w_{n,x}$ when $\{w_n\}_{n\in \mathbb{N}}$ is bounded in $L^2(0,T; H^1_0(0,L))\cap L^\infty(0,T; L^2(0,L))$. To deal with this nonlinear term, we prove that:

\vglue 0.2 cm

\noindent {\sc Claim.} {\it There exists $s>0$ such that
\[
\{w_n\}_{n\in \mathbb{N}} \text{ is bounded in } L^4(0,T; H^s(0,L)),
\]
the embedding $H^s(0,L)\hookrightarrow L^4(0,L)$ being compact.}

\vglue 0.2 cm

In fact, by interpolation we can deduce that $\{w_n\}$ is bounded in
$$
[L^q(0,T; L^2(0,L)),L^2(0,T; H^1_0(0,L))]_{\theta} = L^p(0,T;
[L^2(0,L),H^1_0(0,L)]_{\theta}),
$$
where $\frac{1}{p}=\frac{1-\theta}{q} + \frac{\theta}{2}$\, and
\,$0<\theta<1$. Thus, choosing $q=\infty$, $\theta=1/2$, so that
$p=4$, the claim holds with $s=1/2$, i.e.,
$$[L^2(0,L),H^1_0(0,L)]_{\frac{1}{2}}=H^\frac{1}{2}(0,L).$$
Furthermore, the embedding $H^\frac{1}{2}(0,L)\hookrightarrow
L^4(0,L)$ is compact. Therefore, from of the previous estimates and using a classical compactness result shown by \cite[Corollary 4]{simon}, we can extract a subsequence of $\{w_n\}_{n\in \mathbb{N}}$, still denoted by the same index $n$, such that
\begin{eqnarray}
w_n\rightarrow w \hbox{ strongly in } L^4(0,T; L^4(0,L)),
\end{eqnarray}
which allows us to pass to the limit in the nonlinear term. Then, arguing as in Theorem \ref{main-dec}, we deduce that $E_s(t)$ decays to zero exponentially.
\end{proof}

\subsubsection{Existence of weak solutions}

\begin{definition}
For $(u^{0},v^{0})\in X$ and $T>0$, we denote by a weak solution
of \eqref{1-1a}-\eqref{gg2a} any function $$u\in C_w([0,T];X)\cap L^2(0,T;[H^1(0,L)]^2)$$ which solves \eqref{1-1a}-\eqref{gg2a}, and such that, as $p\rightarrow 4$,
\begin{eqnarray*}
&&u_p \to u\,\,\mbox{weakly}\ast\,\,\mbox{in}\,\,
L^\infty (0,T;X),\\
&&u_p \to u\,\,\mbox{weakly}\,\,\mbox{in}\,\,
L^2(0,T; [H^1(0,L)]^2),
\end{eqnarray*}
$u_p$ denoting a solution of \eqref{1-1a}-\eqref{gg2a} (as given by
Theorem \ref{existence2}) for $a(x)=x^p$ and $2<p<4$.
\end{definition}

\begin{theorem}\label{main-dec2} For any $(u^0,v^0)\in [L^2(0,L)]^2$ there exists a weak solution of
\begin{equation}\label{1-1a2}
\left\{
\begin{array}{l}
\vspace{1mm}u_t+u_{xxx}+a_3v_{xxx}+u^pu_x + a_1vv_x + a_2(uv)_x + Gu =  0 \\
b_1v_t + rv_x + v_{xxx} + b_2a_3u_{xxx} + v^pv_x + b_2a_2uu_x + b_2a_1(uv)_x + Gv = 0,
\end{array}
\right.
\end{equation}
satisfying the following boundary conditions
\begin{equation}\label{gg2a2}
\begin{cases}
u(0,t)=0,\,\,u(L,t)=0,\,\,u_{x}(L,t)=0,\\
v(0,t)=0,\,\,v(L,t)=0,\,\,v_{x}(L,t)=0,
\end{cases}
\end{equation}
where $0<x<L$, $t>0$, $Gu$, $Gv$ are defined by \eqref{operator_new} and \eqref{operator_new_1}, and $2<p<4$.
\end{theorem} 

We follow the same steps of the previous estimates and for the sake of simplicity we drop the notation $u_p$ and use the notation $u$. 

\begin{proof}We will only present a sketch of the proof.

\vglue 0.1 cm

\noindent{\bf Estimate I.} Using the multipliers $xu$ and $xv$ the solution fulfill
\begin{equation}\label{critical-3-1}
\begin{array}{c}
\displaystyle\int_0^T\int_0^L (u^2_x + v^2_x) dx dt \leq
C\left[||(u^{0},v^{0})||_X^2 + \int_0^T\int_0^L (u^{p+2} + v^{p+2}) dx dt
\right],
\end{array}
\end{equation}
where $C=C(T,L)$ is a positive constant. By Gagliardo-Nirenberg inequality we obtain
\begin{equation}\label{critical-4-1}
\begin{array}{l}
\vspace{1mm}\displaystyle\int_0^T\int_0^L u^{p+2} dx dt \leq
C\,\displaystyle\int_0^T \|u(t)\|_{L^2(0,L)}^{(1-\frac{2}{p+2})(p+2)}
\|u_x(t)\|_{L^2(0,L)}^2 dt \\ 
\qquad\qquad\qquad\qquad\leq
C\,\|(u^{0},v^{0})\|^p_X\displaystyle\int_0^T \|u_x(t)\|^2_{L^2(0,L)} dt,
\end{array}
\end{equation}
for some constant $C>0$ that does not depend on $p$. The same holds to 
$\int_0^T\int_0^L v^{p+2} dx dt$. The above estimate and \eqref{critical-3-1} give us that
\begin{equation}\label{critical-5-1}
\begin{array}{c}
(1-||(u^{0},v^{0})||^p_X)||(u,v)||^2_{L^2(0,T;[H^1_0(0,L)]^2)} \leq
C\,||(u^{0},v^{0})||_X^2.
\end{array}
\end{equation}

\vglue 0.1 cm

\noindent{\bf Estimate II.} We are interested to bound the term $u^pu$. By the previous subsection and using \eqref{critical-4-1}-\eqref{critical-5-1}, we deduce that
$$\{u^{p+1}\}\,\, \mbox{is bounded in}\,\,L^{\frac{p+2}{p+1}}((0,T)\times (0,L)),$$
with a bound uniform in $p$. Therefore,
$$\{(p+1)u^p u_x\}=\{\partial_x(u^{p+1})\}\,\, \mbox{is bounded in}
\,\,L^{\frac{p+2}{p+1}}(0,T;H^{-2}(0,L)),$$
i.e.,
$$\{u^p u_x\}=\{\partial_x(u^{p+1})\}\,\, \mbox{is bounded in}
\,\,L^{\frac{p+2}{p+1}}(0,T;H^{-2}(0,L))\subset L^{\frac{6}{5}}(0,T;H^{-2}(0,L)),$$
since $p$ is intended to go to $4$ and $\frac{p+2}{p+1}>\frac{6}{5}$. The same holds to $v^pv_x$.

\vglue 0.1 cm

\noindent{\bf Estimate III.} Combining the equations in \eqref{1-1a2} and the previous estimates, we deduce that
$$(u_t,v_t)\,\,\mbox{is bounded in} \,\,L^{\frac{p+2}{p+1}}(0,T;[H^{-2}(0,L)]^2)\subset L^{\frac{6}{5}}(0,T;H^{-2}(0,L)),$$
with a bound uniform in $p$.

\vglue 0.1 cm

\noindent{\bf Estimate IV.} Finally, to deal with to the nonlinear term we claim  that the following hold:

\vglue 0.2 cm

\noindent {\sc Claim.} {\it There exists $s>0$ such that $\{u_p\}$ is
bounded in $L^4(0,T; H^s(0,L))$, the embedding
$$H^s(0,L)\hookrightarrow L^4(0,L)$$ being compact.}

\vglue 0.2 cm

As made in the previous subsection and using interpolation the results holds. Indeed, since $\frac{1}{p}=\frac{1-\theta}{q} + \frac{\theta}{2}$,
choosing $\theta=\frac{1}{2}$ and 
$q=\infty$ we obtain the claim with $s=\frac{1}{2}$.

Due to the statement above and classical compactness results due to \cite[Corollary 4]{simon}, we can extract a subsequence of $\{u_p\}$, still denoted by $\{u_p\}$, and a function $$u\in L^\infty(0,T; X)\cap L^2(0,T; [H^1(0,L)]^2),$$ such that $ u_p\rightarrow u$, as $p\rightarrow 4$, in the sense described above. Therefore, the result is archived.
\end{proof}

\section{Further comment\label{sec5}}
\subsection{Only one damping mechanism} 

Note that the energy dissipation law \eqref{dissipation}, as well as, \eqref{dec2} is still valid for only one damping mechanism $Gu$ (or $Gv$). So it is natural to believe that the same method developed here should ensures the exponential decay of solution of the system \eqref{1-1a}-\eqref{gg2a} with only $Gu$ (or $Gv$). However, we can not apply directly the ideas contained in the proof of Proposition \ref{local-dec} because that the unique continuation property proved in \cite{nina} it is valid when:  $$(u,v)=0 \text{ in } \omega \text{ implies } (u,v)=0 \text{ in } (0,T)\times(0,L),$$ due the Carleman estimate proved by the authors in \cite[Theorem 3.1]{nina}. Thus, to get the result with one damping mechanism a new Carleman estimate will be need with only one observation in $u$ (or $v$). Therefore, the following problem remains open:

\vglue 0.2 cm

\noindent\textbf{Open problem}\textit{ Is the system \eqref{1-1a}-\eqref{gg2a} with one damping mechanism $Gu$ (or $Gv$) asymptotically stable as $t\rightarrow + \infty$ ?}

\vglue 0.2 cm

\noindent\textbf{Acknowledgments:} The author thank the anonymous referee for their helpful comments and suggestions.

\end{document}